\documentclass[11pt]{article}
\usepackage{amssymb, amsmath, hyperref, array}
\usepackage[dvipsnames]{xcolor}

\def\C{\mathbb{C}}
\def\R{\mathbb{R}}
\def\N{\mathbb{N}}

\def\cC{{\cal C}}
\def\cS{{\cal S}}

\newtheorem{defn}{Definition}
\newtheorem{notn}[defn]{Notation}
\newtheorem{lemma}[defn]{Lemma}
\newtheorem{proposition}[defn]{Proposition}
\newtheorem{theorem}[defn]{Theorem}

\newtheorem{example}[defn]{Example}

           {\vspace{3.3mm}
           \noindent{\bf #1}\it}%
           {\vspace{3.3mm}}

\newenvironment{proof}[1]{
  \trivlist \item[\hskip \labelsep{\it #1}]}{\hfill\mbox{$\square$}
  \endtrivlist}

\title{On sum of squares certificates of non-negativity on a strip}

\date{}

\author{Paula Escorcielo\footnote
{{\scriptsize Partially supported by the Argentinian grants} {\footnotesize UBACYT 20020160100039BA} 
{\scriptsize and} {\footnotesize PIP 11220130100527CO CO\-NI\-CET}.
} 
\qquad \qquad
Daniel Perrucci$^{*}$
\\[3mm]
{\small Departamento de Matem\'atica, FCEN, Universidad de Buenos Aires, Argentina}\\ 
{\small IMAS, CONICET--UBA, Argentina}
}

\begin{document}

\maketitle

\begin{abstract}
In \cite{Marsh}, Murray Marshall proved that every $f \in \R[X,Y]$ non-negative on the strip 
$[0,1] \times \R$ can be written as $f= \sigma_0 + \sigma_1 X(1-X)$
with $\sigma_0, \sigma_1$ sums of squares in $\R[X,Y]$.
In this work, we present a few results concerning this representation in particular 
cases. 
First, 
under the assumption $\deg_Y f \leq 2$, 
by characterizing the extreme rays of a suitable cone, we obtain a degree bound for each term.
Then, we consider the case of $f$ positive 
on $[0,1] \times \R$ and \emph{non-vanishing at infinity}, and 
we show again a degree bound for each term, coming from a 
constructive method to obtain the sum of squares representation. Finally, we show that this constructive 
method 
also works in the case of $f$ having only a finite number of zeros, all of them 
lying on the boundary of the strip, and such that $\frac{\partial f}{\partial X}$ does not vanish at any of them. 
\end{abstract}

\section{Introduction}

Let $g_1, \dots, g_s \in \R[X_1,\dots,X_n]$ and consider
the basic closed semialgebraic set 
$$
S = \{x \in \R^n \ | \  g_1(x) \ge 0, \dots, g_s(x) \ge 0 \}.
$$
Given $f \in \R[X_1,\dots,X_n]$ such that $f$ is non-negative on $S$, a classical 
question is if there is a 
representation of $f$ 
which makes evident this fact.
Concerning this problem, there are two important algebraic objects associated to $g_1, \dots, g_s$:  
the preordering 
$$
T(g_1, \dots, g_s) = \Big\{\sum_{I \subset \{1, \dots, s \}}  \sigma_I \prod_{i \in I}g_i
 \ | \ 
\sigma_I \in \sum \R[X_1, \dots, X_n]^2 \hbox{ for every } I \subset \{1, \dots, s \} \Big\}
$$
and 
the quadratic module 
$$
M(g_1, \dots, g_s) = \left\{\sigma_0 + \sigma_1g_1 + \dots + \sigma_sg_s \ | \ 
\sigma_0, \sigma_1, \dots, \sigma_s \in \sum \R[X_1, \dots, X_n]^2\right\}.
$$
It is clear that $M(g_1, \dots, g_s) \subset T(g_1, \dots, g_s)$, but the equality only holds
in some special cases, for instance when $s = 1$.
It is also clear that every polynomial $f \in T(g_1, \dots, g_s)$ is non-negative on $S$,
but the converse is not true in general (see \cite[Example]{Ste}).

Schm\"udgen Positivstellensatz (\cite{Schm}) states that if $S$ is compact, 
every polynomial $f \in \R[X_1, \dots,X_n]$ positive on $S$ belongs to $T(g_1, \dots, g_s)$. 
On the other hand, Putinar Positivstellensatz (\cite{Put}) states that if $M(g_1, \dots, g_s)$ is archimedean,
every polynomial $f \in \R[X_1, \dots,X_n]$ positive on $S$ belongs to $M(g_1, \dots, g_s)$. 
Recall that the quadratic module $M(g_1, \dots, g_s)$ is archimedean if 
there exists $r \in \N$ such that 
$$
r - X_1^2 - \dots - X_n^2 \in M(g_1, \dots, g_s).
$$
Note that if $M(g_1, \dots, g_s)$ is archimedean, then $S$ is compact, but again, the converse is not true in general 
(see \cite[Example 4.6]{JacPres}).

In the case where $\dim S \ge 3$ or in the case where $n = 2$ and $S$ contains an affine full-dimensional cone,
there exist polynomials non-negative on $S$ which do not belong to $T(g_1, \dots, g_s)$ (\cite{Sche}).
On the contrary, M. Marshall proved in \cite{Marsh} the following result for polynomials non-negative on the strip
$[0,1] \times \R \subset \R^2$:

\begin{theorem} \label{th:Marshall}
 Let $f \in \R[X,Y]$ with $f \geq 0 $ on $[0,1]\times \R$. Then
 \begin{align}\label{repr_f}
 f= \sigma_0 + \sigma_1 X(1-X)
 \end{align}
 with $\sigma_0, \sigma_1 \in \sum \R[X,Y]^2$.
\end{theorem}

In other words, Theorem \ref{th:Marshall} states that every polynomial non-negative on the strip $[0,1]\times \R$ 
belongs to $M(X(1-X))$. 
This result was later extended to other two-dimensional 
semialgebraic sets in \cite{NguPowers} and \cite{SchWen}.

In this paper, we present some results concerning effectivity issues around 
the representation obtained in Theorem \ref{th:Marshall}, in 
particular cases.

For instance, a natural question is if it is possible to bound the degrees of each term  
in (\ref{repr_f}).
In Section \ref{sec:deg_2}, we
prove a degree bound for each term in the case 
$\deg_Y f \leq 2$. To this end, 
we first characterize all the extreme rays of a suitable cone 
containing $f$
and study their representation as in 
(\ref{repr_f}). 
The main result in this section is the following.

\begin{theorem}\label{th:main_degY_2} Let $f \in \R[X,Y]$ 
with $f \ge 0$ on $[0,1]\times \R$ and
$\deg_Y f \le 2$.  
Then $f$ can be written as in (\ref{repr_f}) with 
$$
\deg (\sigma_0),  \deg (\sigma_1 X(1-X)) \le \deg_X f + 3.
$$
 
\end{theorem}

In Section \ref{sec:al_app}, we deal again with the question of bounding the degrees of each term in 
(\ref{repr_f}) in a different situation. First, in Section \ref{subsect:positiv}, 
we consider the case where $f$ is positive on $[0, 1] \times \R$
and 
does not \emph{vanish at infinity}. To make this concept precise, we introduce the following definition 
coming from \cite{Pow}:
\begin{defn}
Let $f \in \R[X,Y]$ and $m=\deg_Y f$. 
The polynomial $f$ is \emph{fully} $m$-\emph{ic} on $[0,1]$ if for every $x \in [0,1]$,  $f(x,Y) \in \R[Y]$ has degree $m$.
 \end{defn}

Given
  $$
  f= \sum_{0 \leq i \leq m} \sum_{0 \leq j \leq d} a_{ji}X^jY^i \in \R[X,Y],
  $$
define
  $$
  \bar f= \sum_{0 \leq i \leq m} \sum_{0 \leq j \leq d} a_{ji}X^jY^iZ^{m-i} \in \R[X,Y, Z].
  $$

Note that if $f>0$ on $[0,1]\times \R$ and $f$ is fully $m$-ic on $[0,1]$ then $m$ is even and 
$\bar f>0$ on 
$\{(x, y, z) \ | \ x \in [0, 1], \, y^2 + z^2 = 1 \}$. 

We note as usual
$$
\| f \|_{\infty} = \max \{  |a_{ji}| \ | \ 0 \le i \le m, \ 0 \le j \le d\}.
$$
We prove the following result. 
 
\begin{theorem}\label{thm:f_positivo}    
  Let $f \in \R[X,Y]$ 
  with $f > 0$ on $[0,1]\times \R$, $f$ fully $m$-ic on $[0, 1]$, $d = \deg_X f \ge 2$ and  
      $$
   f^{\bullet} =  \min\{\bar f(x, y, z) \ | \ x \in [0, 1], \, y^2 + z^2 = 1 \} > 0.
   $$  
  Then $f$ can be written as in (\ref{repr_f}) with 
     $$
  \deg (\sigma_0), \deg(\sigma_1 X(1-X)) \leq \frac{d^3(m+1)\|f\|_{\infty}}{f^{\bullet}}.
  $$
   \end{theorem}

Note that the cases $\deg_X f = 0$ and $\deg_X f = 1$ are not covered by 
Theorem \ref{thm:f_positivo}, but these cases are of a simpler nature.  
If $\deg_X f = 0$,
$f$ belongs to $\R[Y]$ and is non-negative  on $\R$, then $f$ can simply be written as 
a sum of squares in $\R[Y]$ with the degree of each term bounded by $m$ 
(see \cite[Proposition 1.2.1]{Marshall_book} and 
\cite{MagSafSch}).
If $\deg_X f=1$, we have
$$
f(X,Y)= f(1,Y)X+f(0,Y)(1-X)
$$
and, since $f(0,Y)$ and $f(1,Y)$ are non-negative on $\R$, 
again these polynomials can be written as 
sums of squares in $\R[Y]$ with the degree of each term bounded by $m$; 
then, 
using the identities
$$
X = X^2 + X(1-X)
\qquad \hbox{ and } \qquad
1-X = (1-X)^2 + X(1-X),
$$
we take $\sigma_0 = f(1,Y)X^2 + f(0, Y)(1-X)^2$ and 
$\sigma_1 = f(1,Y) + f(0, Y)$ and the identity $f = \sigma_0 + \sigma_1X(1-X)$ holds 
with the degree of each term bounded by $m+2$.

To prove Theorem \ref{thm:f_positivo}, in Section \ref{subsect:positiv} we show a constructive way of producing
the representation in Theorem \ref{th:Marshall}
in the case of $f$ positive on $[0, 1] \times \R$ and fully $m$-ic on $[0, 1]$, 
and then we bound the degrees of each term. 
A similar constructive way of obtaining this representation was already given in 
\cite[Proposition 3]{PowRez} under slightly different hypothesis. 
The idea behind the construction is  
to consider the unbounded variable as a parameter 
and to produce  
a uniform version of 
a representation theorem for the segment $[0,1]$ using the effective version of P\'olya's Theorem 
from \cite{Polya_bound}.
This technique was also used in related problems in \cite{Pow} and \cite{EscPer}.

Finally, in Section \ref{subsect:zeros}, we prove that the constructive method from the previous section 
also works in the case of $f$ non-negative on the strip and having only a finite number of zeros, all of them 
lying on the boundary, 
and such that $\frac{\partial f}{\partial X}$ does not vanish at any of them.

\section{The case $\deg_Y f\leq 2$}\label{sec:deg_2}
In this section we
consider the problem of finding a degree bound for the representation in Theorem \ref{th:Marshall} 
under the assumption $\deg_Y f \le 2$. 
Since it will be more convenient to homogenize with respect to the unbounded variable, 
we introduce the  set 
$$\cS = [0, 1] \times ( \R^2 \setminus \{(0, 0)\} ) \subseteq  \R^3.
$$
It is easy to see that for $\bar f=f_2(X)Y^2+f_1(X)YZ+f_0(X)Z^2$ non-negative on $\cS$ and $x_0 \in [0, 1]$, 
$f_2(x_0) \ge 0$ and $f_0(x_0) \ge 0$ and either $f(x_0, Y, Z) = 0$ or 
$\deg_Y f(x_0, Y, Z)$ and $\deg_Z f(x_0, Y, Z)$ are even numbers; 
therefore, if $X - x_0 \, | \,  f_2$ or $X-x_0 \, | \, f_0$, then $X-x_0 \, | \, f_1$.
Moreover, if $x_0 \in (0, 1)$ and $X - x_0 \, | \, f_2$, then $(X - x_0)^2 \, |\, f_2$. Similarly, 
if $x_0 \in (0, 1)$ and $X - x_0 \, |\, f_0$, then $(X - x_0)^2 \,|\, f_0$.

We introduce the following cone. 

\begin{defn}
Given $d, e \in \N_0$, we define 
$$
\cC_{d, e} 
= \Big\{ 
\bar f=f_2(X)Y^2+f_1(X)YZ+f_0(X)Z^2 \in \R[X,Y, Z]
\ | \ 
$$
$$
\bar f \ge 0 \hbox{ on } \cS, \
\deg f_2 \leq d, \ \deg f_1 \leq \left \lfloor \frac12(d + e) \right \rfloor, \ \deg f_0 \leq e
\Big\}.
$$
\end{defn}

We can think of
$\cC_{d,e}$ as included in $\R^{d + \lfloor \frac12(d + e) \rfloor + e + 3}$ by identifying 
each $\bar f \in \cC_{d,e}$ with its vector of coefficients in some prefixed order.
It is easy to see that $\cC_{d, e}$ is a closed cone which does not contain lines. Therefore, we can use the 
following well-known result (see for instance \cite[Section 18]{Rockafellar}).

\begin{theorem}\label{th:extr_rays} 
Let $\cC \subseteq \R^N$ be a closed cone which does not contain lines, then
every element of $\cC$ can be written as a sum of elements lying on extreme rays of $\cC$.
\end{theorem}

For a given $f \in \R[X, Y]$ non-negative on $[0, 1] \times \R$, the strategy for proving
that Theorem \ref{th:main_degY_2} holds for $f$ is to use the classical idea of
characterizing the extreme rays of $\cC_{d,e}$, then to study 
the homogenized representation
as in Theorem \ref{th:Marshall} for the elements lying on these rays, 
and finally to
decompose $\bar f$ as 
a sum of them. 

Under the additional hypothesis that $d$ and $e$ have the same parity, our characterization of the extreme rays of 
$\cC_{d, e}$ is the following.

\begin{theorem}\label{th:main}
Let $d, e \in \N_0$ such that $d\equiv e \, (2)$.
The extreme rays of $\cC_{d, e}$
are the rays generated by the polynomials of the form $r(X)(p(X)Y + q(X)Z)^2$
with
\begin{itemize}
\item $p$ and $q$ not simultaneously zero  and $(p: q) = 1$, 
\item $r \ne 0$,  $r \ge 0$ on $[0,1]$ and $r$ with $\deg r$ real roots in $[0, 1]$ (counted with multiplicity), 
\item $2 \deg p \le d, 2 \deg q \le e$ and $\deg r = \min \{d - 2\deg p, e - 2\deg q \}.$
\end{itemize}
\end{theorem}

To prove Theorem \ref{th:main}, the idea is to proceed inductively on a 
sequence of cones ordered \emph{by inclusion}. To do so, we need to show first that
given $\bar f=f_2(X)Y^2+f_1(X)YZ+f_0(X)Z^2 \in \cC_{d,e}$
some factors of $f_2(X)$ or $f_0(X)$ are necessarily also factors of $f_1(X)$; 
in this case,
after removing these factors we move to a smaller cone.

The following lemmas are some basic auxiliary results concerning extreme rays of $\cC_{d, e}$.

\begin{lemma} \label{lem:se_anula}
Let $d, e \in \N_0$ and let $\bar f$ be a generator of an extreme ray of $\cC_{d, e}$. 
Then $\bar f$ vanishes at some point of $\cS$. 
\end{lemma}
\begin{proof}{Proof:}
Suppose $\bar f > 0$ on $\cS$ and take
$$
c = \min\{\bar f(x, y, z) \, | \, x \in [0, 1], \, y^2 + z^2 = 1 \} > 0.
$$
Consider $cY^2$, $c(Y^2+Z^2) \in \cC_{d,e}$. We have 
$$
0 \le cY^2 \le c(Y^2 + Z^2) \le \bar f \ \hbox{ on } \cS,
$$ 
but since $\bar f$ generates an extreme ray of $\cC_{d, e}$, $\bar f$ is a scalar multiple of both $cY^2$ and $c(Y^2 + Z^2)$ which is impossible.
\end{proof}

\begin{lemma} \label{lem:algun_coef_cero}
Let $d, e \in \N_0$ and let 
$\bar f = f_2(X)Y^2 + f_1(X)YZ + f_0(X)Z^2$  be a generator of an extreme ray of $\cC_{d, e}$. 
If 
$f_2 = 0$, $f_1 = 0$ or $f_0 = 0$, then 
$\bar f$ is of the form 
$$
r(X)Y^2 \hbox{ or } r(X)Z^2.
$$
\end{lemma}

\begin{proof}{Proof:} 
If $f_2 = 0$ then $f_1 = 0$, $\bar f = f_0(X)Z^2$ and we take $r(X) = f_0(X)$. Similarly, if $f_0 = 0$ 
then $f_1 = 0$, $\bar f = f_2(X)Y^2$ and we take $r(X) = f_2(X)$.
On the other hand, if $f_1 = 0$ and $f_2, f_0 \ne 0$, then 
$$
0 \le f_2(X)Y^2 \le f_2(X)Y^2 + f_0(X)Z^2 = \bar f \ \hbox{ on } \cS
$$
which, proceeding similarly to the proof of Lemma \ref{lem:se_anula}, is impossible. 
\end{proof}

The following lemma shows that the second and third condition in the characterization of the extreme rays 
in Theorem \ref{th:main} are indeed consequences of the first condition. 

\begin{lemma} \label{lem:conclu_1_var}
Let $d, e \in \N_0$. 
If  $r(X)(p(X)Y + q(X)Z)^2$ 
with $p$ and $q$ not simultaneously zero and $(p: q) = 1$
generates an extreme ray of $\cC_{d, e}$,
then 
\begin{itemize}
\item $r \ne 0$,  $r \ge 0$ on $[0,1]$ and $r$ has $\deg r$ real roots in $[0, 1]$ (counted with multiplicity), 
\item $2 \deg p \le d, 2 \deg q \le e$ and $\deg r = \min \{d - 2\deg p, e - 2\deg q \}.$
\end{itemize}
\end{lemma}

\begin{proof}{Proof:}
Let $\bar f = r(X)(p(X)Y + q(X)Z)^2$. Since $\bar f \ne 0$, $r \ne 0$, and 
since $\bar f \geq 0$ on $\cS$, $r \geq 0$ on $[0,1]$. If $r$ has a complex non-real root, or a 
real root which does not belong to the
interval $[0,1]$, it is easy to see that $r$ can be written as $r=r_1+r_2$
with $r_1,r_2\in \R[X]-\{0\}$, 
$\deg r_1, \deg r_2 \le \deg r$,
$\deg r_1 \ne \deg r_2$ and
$r_1,r_2 \geq 0$ on $[0,1]$. 
Then for $i = 1, 2$, we take $f_i=r_i(X)(p(X)Y + q(X)Z)^2 \in \cC_{d, e}$ and we have 
$$
0 \le f_i \le \bar f \ \hbox{ on } \cS,
$$
but since $\bar f$ generates an extreme ray of $\cC_{d, e}$, $\bar f$ is a scalar multiple of both $f_1$ and $f_2$ which is impossible.

Since $\bar f \in \cC_{d,e}$, we have 
$2 \deg p \le d, 2 \deg q \le e$ and $\deg r \le \min \{d - 2\deg p, e - 2\deg q \}$.
If $\deg r < \min \{d - 2\deg p, e - 2\deg q \}$, we have 
$X \bar f \in \cC_{d,e}$ and 
$$
0 \le X \bar f \le \bar f \ \hbox{ on } \cS
$$
which is again impossible for similar reasons. 
\end{proof}

In order to prove Theorem \ref{th:main}, we will do several changes of variables. 
The following three
lemmas summarize the properties we need. We omit their proofs since they are very simple.

\begin{lemma} \label{lem:camb_variable}
Let $d, e \in \N_0$ with $d \le e$,  $\bar f \in  \cC_{d, e}$, $\beta \in \R$
and
$h 
\in \R[X, Y, Z]$
defined by 
$$
h(X, Y, Z) =\bar f(X, Y + \beta Z, Z)= f_2(X)Y^2 + h_1(X)YZ + h_0(X)Z^2.
$$
Then:
\begin{itemize}
 \item $h$ belongs to $\cC_{d, e}$. 
 
 \item If $\bar f$ generates an extreme ray of $\cC_{d, e}$, 
 then $h$ generates an extreme ray of $\cC_{d,e}$.
 
 \item If $(x_0, y_0, z_0) \in \cS$ with $z_0 \ne 0$ and $\bar f(x_0, y_0, z_0) = 0$ 
 and $\beta = y_0/z_0$, then $h_0(x_0) = 0$.
 
 \item If $h$ can be written as
 $
 r(X)(p(X)Y + q(X)Z)^2
 $
 with $p$ and $q$ not simultaneously zero and $(p:q)$ $= 1$, 
 then
 $\bar f$ 
 can be written as
 $$
 r(X)(p(X)Y + (-\beta p(X) + q(X))Z)^2
 $$
 with $p$ and $ -\beta p + q$ not simultaneously zero and 
 $(p : - \beta p + q) = 1$. 

\end{itemize}
\end{lemma}

\begin{lemma} \label{lem:camb_variable_d_men_e}
Let $d, e \in \N_0$ with $d +2 \le e$,  
$\bar f  \in \cC_{d, e}$, 
$\ell \in \R[X]$ with $\deg \ell = 1$ and 
$h \in \R[X, Y, Z]$ defined by 
$$
h(X, Y, Z) = \bar f(X, Y + \ell(X) Z, Z)
= f_2(X)Y^2 + h_1(X)YZ + h_0(X)Z^2. 
$$
Then: 
\begin{itemize}
\item $h$ belongs to $\cC_{d, e}$.

\item If $\bar f$ generates an extreme ray of $\cC_{d, e}$, 
then  $h$  generates an extreme ray of $\cC_{d, e}$.

\item If $(x_0, y_0, z_0),$ $(x_1, y_1, z_1)
\in \cS$ 
with $x_0 \ne x_1$, $z_0, z_1 \ne 0$, $y_0/z_0 \ne y_1/z_1$
and $f(x_0, y_0, z_0) = f(x_1, y_1, z_1) = 0$ and
$$
\ell(X) = \frac{y_1/z_1 - y_0/z_0}{x_1 - x_0}(X - x_0) + y_0/z_0,
$$
then $h_0(x_0) = h_0(x_1) = 0$. 

\item  If $h$ can be written as 
$
r(X)(p(X)Y + q(X)Z)^2
$
with $p$ and $q$ not simultaneously zero and $(p:q)$ $= 1$, 
then $\bar f$ can be written as
$$
r(X)(p(X)Y + (-\ell(X) p(X) + q(X))Z)^2
$$
with $p$ and $ -  \ell p + q$ not simultaneously zero and 
$(p : - \ell p + q) = 1$. 
\end{itemize}
\end{lemma}

\begin{lemma} \label{lem:camb_variable_d_ig_e}
Let $d, e \in \N_0$ with $d = e$,  $\bar f \in \cC_{d, e}$,
$\beta_0, \beta_1 \in \R$
with $\beta_0 \ne \beta_1$ and 
$h \in \R[X, Y, Z]$ defined by 
$$
h(X, Y, Z) = f(X, \beta_0 Y + \beta_1 Z, Y + Z) 
= h_2(X)Y^2 + h_1(X)YZ + h_0(X)Z^2. 
$$
Then:

\begin{itemize}

\item $h$ belongs to $\cC_{d, e}$.

\item If $\bar f$ generates an extreme ray of $\cC_{d, e}$, 
then $h$ generates an extreme ray of $\cC_{d, e}$.

\item If $(x_0, y_0, z_0),$
$(x_1, y_1, z_1)
\in \cS$ 
with $z_0, z_1 \ne 0$, $y_0/z_0 \ne y_1/z_1$  and $f(x_0, y_0, z_0) = f(x_1, y_1, z_1) = 0$ and  
$\beta_0 = y_0/z_0$, 
$\beta_1 = y_1/z_1$, then $h_2(x_0) = h_0(x_1) = 0$.

\item If $h$ can be written as 
$
r(X)(p(X)Y + q(X)Z)^2
$
with $p$ and $q$ not simultaneously zero and $(p:q)$ $= 1$, 
then $\bar f$ 
can be written as 
$$
\frac1{(\beta_0 - \beta_1)^2}r(X)((p(X) - q(X))Y + (-\beta_1 p(X) + \beta_0q(X))Z)^2
$$
with $p- q$ and $-\beta_1 p + \beta_0q $ not simultaneously zero and 
$(p- q :-\beta_1 p + \beta_0q ) = 1$. 
\end{itemize}
\end{lemma}

We are ready to prove the characterization of the extreme rays of the 
cone $\cC_{d, e}$ given in Theorem \ref{th:main}.

\begin{proof}{Proof of Theorem \ref{th:main}:}
We begin by proving that if $\bar f = r(X)(p(X)Y + q(X)Z)^2$ 
with $r, p$ and $q$ as in the statement of Theorem \ref{th:main}, then 
$\bar f$ generates an extreme ray of $\cC_{d,e}$.
Consider
$$
g = g_2(X)Y^2 + g_1(X)YZ + g_0(X)Z^2 \in \cC_{d, e}
$$
such that $0 \le g \le \bar f$ on $\cS$. We want to show that $g$ is a scalar multiple of $\bar f$. 

If $p=0$, since $(p:q) = 1$ we have $q = \lambda \in \R\setminus\{0\}$ and  then $\deg r = e$. On the other hand, 
for every $x \in [0,1]$, $\bar f(x,1,0)=0$. Then, for every $x \in [0,1]$, $g_2(x)=g(x,1,0)=0$
and this implies $g_2 = g_1=0$. Therefore, $g=g_0(X)Z^2$, but 
since $0 \le g \le \bar f$ on $\cS$, $0 \leq g_0 \leq  \lambda^2 r $ on $[0,1]$. It is easy to see  
that every root of $r$ is necessarily also a root of $g_0$ with at least the same multiplicity, then we have
$\deg r \le \deg g_0 \leq e=\deg r$, $g_0$ is a scalar multiple of $r$ and $g$ is a scalar multiple of $\bar f$.

If $p \neq 0$, we consider $G \in \R[X,Y,Z]$ defined by
$$
G(X,Y,Z)=p(X)^2g(X,Y,Z)=g_2(X)(p(X)Y+q(X)Z)^2+G_1(X)YZ+G_0(X)Z^2.
$$
We first see that $G_1=G_0=0$. 
Take $x_0 \in [0,1]$ such that $p(x_0)\neq 0$. Since $\bar f(x_0, -q(x_0), p(x_0))=0$, 
$G(x_0, -q(x_0), p(x_0))=0$ and then 
\begin{align}\label{id_G_1}
-G_1(x_0)q(x_0)p(x_0)+G_0(x_0)p(x_0)^2=0.
\end{align}
Moreover, since $G \geq 0$ on $\cS$, 
\begin{align}\label{id_G_0}
\frac{\partial G}{\partial Y}(x_0, -q(x_0), p(x_0))= G_1(x_0)p(x_0)=0.
\end{align}
We conclude from (\ref{id_G_1}) and (\ref{id_G_0}) that $G_1(x_0)= G_0(x_0)=0$. This implies $G_1=G_0=0$
and then $p(X)^2g(X,Y,Z)=g_2(X)(p(X)Y+q(X)Z)^2$. Since $(p:q)=1$, $p^2 \, | \,  g_2$ and $g=\tilde g_2(X)(p(X)Y+q(X)Z)^2$
for $\tilde g_2 = g_2/p^2 \in \R[X]$.
Reasoning similarly to the case $p =0$, we see that $\tilde g_2$ is a scalar multiple of $r$ and $g$
is a scalar multiple of $\bar f$.

Now we prove that if $\bar f = f_2(X)Y^2 + f_1(X)YZ + f_0(X)Z^2$ generates an extreme ray of $\cC_{d,e}$ then 
$\bar f$ can be written as in the statement of Theorem \ref{th:main}. To do so, 
we use inductive arguments, considering the families of cones ordered \emph{by inclusion}, this is 
to say, 
$$
\cC_{d_1, e_1} \le \cC_{d_2, e_2} \qquad \text{ if } \qquad d_1 \le d_2\text{ and } e_1 \le e_2.
$$
Actually, for $(d,e)=(0,0)$, the result is easy to check using Lemma \ref{lem:se_anula}, so from now on
we assume $(d, e) \ne (0, 0)$. 
Using Lemma \ref{lem:algun_coef_cero} and Lemma \ref{lem:conclu_1_var}, we can assume $f_2, f_1, f_0 \ne 0$. 

First, we prove the result in two particular cases. 

\begin{enumerate}

\item[A1.]  There is $x_0 \in [0, 1]$ such that $(X-x_0)^2 \, | \, f_2$ or $(X-x_0)^2 \, | \, f_0$:

Without loss of generality, suppose $(X-x_0)^2 \, | \, f_2$, then $X-x_0 \, | \, f_1$.
Consider $h_2 = f_2/(X-x_0)^2, \, h_1 = f_1/(X-x_0) \in \R[X]$  and
$$
h = h_2(X)Y^2 + h_1(X)YZ + f_0(X)Z^2 \in \R[X, Y, Z],
$$
then 
$$
h(X, (X-x_0)Y, Z) = \bar f(X, Y, Z) \qquad \hbox{ and } \qquad
h(X, Y, Z) =\bar  f\left(X, \frac{Y}{X-x_0}, Z\right)
$$

Note that $h \in \cC_{d-2, e}$. Indeed, 
$h$ verifies the degree bounds
and 
$h \ge 0$ on  $\{(x, y, z) \in  \cS \ | \ x\ne x_0 \}$,
by continuity,  $h \ge 0$ on $\cS$.  
In order to apply the inductive hypothesis, let us prove that $h$ generates an extreme ray of $\cC_{d-2, e}$. 
Given
$$
g = g_2(X)Y^2 + g_1(X)YZ + g_0(X)Z^2 \in \cC_{d-2, e}
$$ 
such that $0 \le g \le h$ on $\cS$,  we consider
$$
\tilde g = (X-x_0)^2g_2(X)Y^2 + (X-x_0)g_1(X)YZ + g_0(X)Z^2 \in \R[X, Y, Z],
$$ 
since $\tilde g(X, Y, Z) = g(X, (X-x_0)Y, Z)$, $\tilde g \in \cC_{d,e}$ and $0 \le \tilde g \le \bar f$ on $\cS$. 
Therefore, $\tilde g$ is a scalar multiple of $\bar f$ and $g$ is a scalar multiple of $h$. 

By the inductive hypothesis, $h$ is of the form
$$
h(X, Y, Z) = \tilde r(X)(\tilde p(X)Y + \tilde q(X)Z)^2
$$
with $\tilde p$ and $\tilde q$ not simultaneously zero and $(\tilde p: \tilde q) = 1$. 
Then, 
$$
\bar f(X, Y, Z)  = \tilde r(X)((X-x_0) \tilde p(X)Y + \tilde q(X)Z)^2.
$$
If $X - x_0 \, \not| \, \tilde q$, we take $r = \tilde r$, $p = (X-x_0) \tilde p$ and $q = \tilde q$, and 
if $X - x_0 \, | \, \tilde q$, we take $r = (X-x_0)^2 \tilde r$, $p = \tilde p$ and $q = \tilde q /(X-x_0) \in \R[X]$.
In both cases we have $(p:q) = 1$ and we conclude using Lemma \ref{lem:conclu_1_var}.

\item[A2.] There is $x_0 \in [0, 1]$ such that $X-x_0 \, | \, f_2,  f_0$:

It is clear that $X-x_0 \, | \, f_1$.
If $x_0 \in (0, 1)$ it is easy to see that  $(X-x_0)^2 \, | \, f_2$ and then we are in case A1, 
so we can suppose $x_0 \in \{0, 1\}$. Without loss of generality assume $x_0 = 0$.
Consider $h = \bar f/X \in \R[X, Y, Z]$. 
Proceeding as in case A1, it is easy to see that $h$ generates an extreme ray of $\cC_{d-1, e-1}$, 
and using the inductive hypothesis we have $h$ is of the form
$$
h(X, Y, Z) = \tilde r(X)(\tilde p(X)Y + \tilde q(X)Z)^2
$$
with $\tilde p$ and $\tilde q$ not simultaneously zero and $(\tilde p: \tilde q) = 1$. 
Then 
we take $r = X\tilde r$, $p = \tilde p$ and $q = \tilde q$ and we conclude using Lemma \ref{lem:conclu_1_var}.
\end{enumerate}

We consider now an auxiliary list of cases in which we prove the result by reducing to cases A1 and A2.

\begin{enumerate}

\item[B1.] There are $x_0 \in \{0,1\}$ and $(y_0,z_0) \in \{(1,0), (0,1)\}$ such that
$\bar f(x_0,y_0,z_0)=0$ and $\bar f(x, y, z) \ne 0$ for every $(x, y , z) \in \cS$ with $x \ne x_0$:

Without loss of generality, suppose $\bar f(0,1,0)=0$, then $f_2(0) = 0$ and $X \, | \, f_1$.
If $X^2 \, | \, f_2$ we are in case A1 and if $X \, | \, f_0$ we are in case A2.  
Moreover, if 
there is $x \in (0, 1]$ with $f_2(x) = 0$, then 
$\bar f(x, 1, 0) = 0$ which  contradicts the hypothesis. 
Similarly, if 
there is $x \in (0, 1]$ with $f_0(x) = 0$, then 
$\bar f(x, 0, 1) = 0$ which also contradicts the hypothesis. 
So from now on we assume $X^2\nmid f_2$, $f_2 > 0$ on $(0, 1]$ and $f_0 > 0$ on $[0, 1]$.

Consider $g_2 = f_2/X, g_1 = f_1/X \in \R[X]$
and note that $g_2 > 0$ in $[0, 1]$.
Since $\bar f(x,y,z) > 0$ for $(x,y,z) \in \cS$ with $x \in (0,1]$, 
$$
f_1(x)^2-4f_2(x)f_0(x)=  x^2g_1^2(x) - 4xg_2(x)f_0(x) < 0, 
$$
for $x \in (0, 1]$, and then 
$$
 xg_1^2(x) - 4g_2(x)f_0(x) < 0
$$
for  $x \in (0, 1]$, but since $g_2(0) > 0$ and $f_0(0) >0$, this last inequality can be extended to $x \in [0, 1]$. 
We take $\varepsilon > 0$ such that 
$$
\frac{xg_1^2(x)}{4g_2(x)} - f_0(x) \le -\varepsilon
$$
for $x \in [0,1]$.
Therefore,
$$
f_1(x)^2-4f_2(x)(f_0(x) - \varepsilon)=  x^2g_1^2(x) - 4xg_2(x)(f_0(x) - \varepsilon) \le 0
$$
for $x \in [0, 1]$.
Let
$
h = f_2(X)Y^2 + f_1(X)YZ + (f_0(X) - \varepsilon)Z^2 \in \R[X,Y,Z].
$
It follows easily that $h \in \cC_{d,e}$ and  $0 \le h \le \bar f$ on $\cS$, 
but then $h$ is a scalar multiple of $\bar f$ which is impossible. 

\item[B2.] There is $(y_0,z_0) \in \{(1,0), (0,1)\}$ such that
$\bar f(0,y_0,z_0) =\bar f(1, y_0,z_0) = 0$ and $\bar f(x, y, z) \ne 0$ for every $(x, y , z) \in \cS$ with $x \in (0, 1)$:

Without loss of generality, suppose $\bar f(0, 1, 0) =\bar f(1, 1, 0) = 0$, then $f_2(0) = f_2(1) = 0$
and therefore $X \, | \, f_1$ and  $X-1 \, | \, f_1$.
If $X^2 \, | \, f_2$ or $(X-1)^2 \, | \, f_2$ we are in case A1 and  
if $X \, | \, f_0$ or $X - 1 \, | \, f_0$ we are in case A2. 
Moreover, if 
there is $x \in (0, 1)$ with $f_2(x) = 0$, 
then $\bar f (x, 1, 0) = 0$ which contradicts the hypothesis. 
Similarly, if
there is $x \in (0, 1)$ with $f_0(x) = 0$, then $\bar f (x, 0, 1) = 0$ which also contradicts the hypothesis. 
So from now on we assume $X^2 \nmid f_2$, $(X-1)^2 \nmid f_2$, $f_2 > 0$ on $(0, 1)$ and $f_0 > 0$ on $[0, 1]$.

Consider $g_2 = f_2/(X(X-1)), g_1 = f_1/(X(X-1)) \in \R[X]$ 
and note that $g_2 < 0$ in $[0, 1]$.
Since $\bar f(x,y,z) > 0$ for $(x,y,z) \in \cS$ with $x \in (0,1)$, 
$$
f_1(x)^2-4f_2(x)f_0(x)=  x^2(x-1)^2g_1^2(x) - 4x(x-1)g_2(x)f_0(x) < 0, 
$$
for $x \in (0, 1)$, and then
$$
 x(x-1)g_1^2(x) - 4g_2(x)f_0(x) > 0
$$
for  $x \in (0, 1)$, but since $g_2(0) <0, g_2(1) < 0, f_0(0) >0$ and $f_0(1)>0$, 
this last inequality can be extended to $x \in [0,1]$. 
We take $\varepsilon > 0$ such that 
$$
\frac{x(x-1)g_1^2(x)}{4g_2(x)} - f_0(x) \le -\varepsilon
$$
for $x \in [0,1]$.
The proof is finished using the same arguments as in case B1.

\item[B3.] There are $(y_0,z_0), (y_1,z_1) \in \{(1,0), (0,1)\}$, $(y_0,z_0)\neq (y_1,z_1)$  such that
$\bar f(0, y_0,z_0) =\bar f(1, y_1,z_1) = 0$  and $\bar f(x, y, z) \ne 0$ for every $(x, y , z) \in \cS$ with $x \in (0, 1)$:

Without loss of generality, suppose $f(0, 1, 0) = f(1, 0, 1) = 0$, then $f_2(0) = f_0(1) = 0$
and therfore $X \, | \, f_1$ and 
$X-1 \, | \, f_1$.
If  $X^2 \, | \, f_2$ or $(X-1)^2 \, | \, f_0$ we are in case A1
and if  $X \, | \, f_0$ or $X-1 \, | \, f_2$ we are in case A2.
Moreover, if 
there is $x \in (0, 1)$ with $f_2(x) = 0$, 
then $\bar f (x, 1, 0) = 0$ which contradicts the hypothesis. 
Similarly, if
there is $x \in (0, 1)$ with $f_0(x) = 0$, then $\bar f (x, 0, 1) = 0$ which also contradicts the hypothesis. 
So from now on we assume $X^2 \nmid f_2$, $(X-1)^2 \nmid f_0$, $f_2 > 0$ on $(0, 1]$ and $f_0 > 0$ on $[0, 1)$.

Consider $g_2 = f_2/X$, $g_1 = f_1/(X(X-1))$, $g_0 = f_0/(X-1) \in \R[X]$
and note that $g_2 > 0$ in $[0, 1]$ and $g_0 < 0$ in $[0, 1]$.
Since $\bar f(x,y,z)>0$ for $(x,y,z) \in \cS$ with $x \in (0,1)$, 
$$
f_1(x)^2-4f_2(x)f_0(x)=  x^2(x-1)^2g_1^2(x) - 4x(x-1)g_2(x)g_0(x) < 0 
$$
for $x \in (0, 1)$, and then
$$
 x(x-1)g_1^2(x) - 4g_2(x)g_0(x) > 0
$$
for  $x \in (0, 1)$, but since $g_2(0) >0, g_2(1) > 0, g_0(0) <0$ and $g_0(1) < 0$, this last inequality can 
be extended to $x \in [0, 1]$. 
We take $\varepsilon > 0$ such that 
$$
\frac{x(x-1)g_1^2(x)}{4g_2(x)} - g_0(x) \ge \varepsilon
$$
for $x \in [0,1]$.
Therefore,
$$
f_1(x)^2-4f_2(x)(x-1)(g_0(x) + \varepsilon)=  x^2(x-1)^2g_1^2(x) - 4x(x-1)g_2(x)(g_0(x) + \varepsilon) \le 0
$$
for $x \in [0, 1]$.
Let
$
h = f_2(X)Y^2 + f_1(X)YZ + (X-1)(g_0(X) + \varepsilon)Z^2 \in \R[X,Y,Z].
$
It follows easily that $h \in \cC_{d,e}$ and  $0 \le h \le \bar f$ on $\cS$, but 
then $h$ is a scalar multiple of $\bar f$ which is impossible. 

\end{enumerate}

We prove now the general case. 
Without loss of generality we suppose $d \le e$.
By Lemma \ref{lem:se_anula}, $\bar f$ vanishes at some point of $\cS$. 
To prove the result we are going to consider three final cases.

\begin{enumerate}

\item[C1.] There is $(x_0, y_0, z_0)  \in \cS$ with $x_0 \in (0, 1)$ such that 
$\bar f(x_0, y_0, z_0) = 0$: 

If $z_0 = 0$, $X - x_0 \, | \, f_2$, then $(X - x_0)^2 \, | \, f_2$ and we are in case A1. 
If $z_0 \ne 0$ we take $\beta = y_0/z_0$ and 
consider
$$
h(X,Y,Z) =\bar f(X, Y + \beta Z, Z) =  f_2(X)Y^2 + h_1(X)YZ + h_0(X)Z^2. 
$$
By Lemma \ref{lem:camb_variable}, $h$ generates an extreme ray of $\cC_{d, e}$ and verifies $h_0(x_0) = 0$.
Then $(X - x_0)^2 \, | \, h_0$ and by case A1 applied to $h$ and 
Lemma \ref{lem:camb_variable} the result follows.

\item[C2.] There are $x_0 \in \{0,1\}$ and  $(y_0, z_0) \in  \cS$ such that 
$\bar f(x_0, y_0, z_0) = 0$ and $\bar f(x, y, z) \ne 0$ for every $(x, y, z) \in \cS$  with $x \ne x_0$:

Without loss of generality, suppose $x_0=0$.
If $z_0 = 0$, we can assume $y_0 = 1$ and we are in case B1. 
If $z_0 \ne 0$, we take $\beta = y_0/z_0$ and consider
$$
h(X,Y,Z) =\bar f(X, Y + \beta Z, Z) =  f_2(X)Y^2 + h_1(X)YZ + h_0(X)Z^2. 
$$
By Lemma \ref{lem:camb_variable},
$h$
generates an extreme ray of $\cC_{d, e}$ and verifies $h_0(0) = 0$ and $h(0, 0, 1) = 0$. 
In addition, 
$h(x, y, z) \ne 0$ for every $(x, y, z) \in \cS$ with $x \ne 0$. 
By case B1 applied to $h$ and 
Lemma \ref{lem:camb_variable} the result follows.

\item[C3.] There are $(y_0,z_0), (y_1,z_1) \in  \cS$ such that 
$\bar f(0, y_0, z_0) =\bar f(1, y_1, z_1)= 0$ 
and $\bar f(x, y, z) \ne 0$ for every $(x, y, z) \in \cS$  with $x \in (0, 1)$:

If $z_0 = z_1= 0$, we can assume $y_0 = y_1 = 1$  and we are in case B2. 

If $z_0 \ne 0$ and $z_1 =  0$, we take $\beta = y_0/z_0$
and consider
$$
h(X,Y,Z) = \bar f(X, Y + \beta Z, Z) =  f_2(X)Y^2 + h_1(X)YZ + h_0(X)Z^2. 
$$
By  Lemma \ref{lem:camb_variable},
$h$
generates an extreme ray of $\cC_{d, e}$ and verifies $h_0(0) = 0$ and $h(0, 0, 1) = 0$. 
On the other hand,
since $\bar f(1, y_1, 0) = 0$, $f_2(1) = 0$ and   
$h(1, 1, 0) = 0$. In addition,
$h(x, y, z) \ne 0$ for every $(x, y, z) \in \cS$ with $x \in (0,1)$. 
By case B3 applied to $h$ 
and 
Lemma \ref{lem:camb_variable} the result follows.
If $z_0 = 0$ and $z_1 \ne 0$ we proceed similarly to the case $z_0 \ne 0$ and $z_1 = 0$.  

The final case is $z_0,z_1 \ne 0$, but we need to split it in three cases. 

If $z_0, z_1 \ne 0$ and $y_0/z_0 = y_1/z_1$, 
we take $\beta = y_0/z_0$ and consider
$$
h(X,Y,Z) =\bar f(X, Y + \beta Z, Z) =  f_2(X)Y^2 + h_1(X)YZ + h_0(X)Z^2. 
$$
By Lemma \ref{lem:camb_variable},
$h$ generates an extreme ray of $\cC_{d, e}$ and verifies $h_0(0) = h_0(1) = 0$, then
$h(0, 0, 1) = h(1, 0, 1) = 0$. 
In addition, 
$h(x, y, z) \ne 0$ for every $(x, y, z) \in \cS$ with $x \in (0,1)$.
By case B2 applied to $h$ and 
Lemma \ref{lem:camb_variable} the result follows.

If $z_0, z_1 \ne 0$ with $y_0/z_0 \ne y_1/z_1$ and $d = e$, 
we take $\beta_0 = y_0/z_0$ and $\beta_1 = y_1/z_1$
and consider
$$
h(X,Y,Z) =\bar f(X, \beta_0 Y + \beta_1 Z, Y + Z) =  h_2(X)Y^2 + h_1(X)YZ + h_0(X)Z^2. 
$$
By Lemma \ref{lem:camb_variable_d_ig_e},
$h$ generates an extreme ray of $\cC_{d, e}$ and verifies $h_2(0) = h_0(1) = 0$, then 
$h(0, 1, 0) = h(1, 0, 1) = 0$.
In addition, 
$h(x, y, z) \ne 0$ for every $(x, y, z) \in \cS$ with $x \in (0,1)$.
By case B3 applied to $h$ and 
Lemma \ref{lem:camb_variable_d_ig_e} the result follows.

Finally, 
if $z_0, z_1 \ne 0$ with $y_0/z_0 \ne y_1/z_1$ 
and $d < e$, since $d\equiv e(2)$, $d +2 \le e$.
Then, we take
$$
\ell(X) = (y_1/z_1 - y_0/z_0)X + y_0/z_0
$$
and consider 
$$
h(X,Y,Z) =\bar f(X, Y + \ell(X) Z, Z) =  f_2(X)Y^2 + h_1(X)YZ + h_0(X)Z^2. 
$$
By Lemma \ref{lem:camb_variable_d_men_e}, $h$
generates an extreme ray of $\cC_{d, e}$ and verifies $h_0(0) = h_0(1) = 0$, 
then, $h(0, 0, 1) = h(1, 0, 1) = 0$.
In addition, 
$h(x, y, z) \ne 0$ for every $(x, y, z) \in \cS$ with $x \in (0,1)$.
By case B2 applied to $h$ and 
Lemma \ref{lem:camb_variable_d_men_e} the result follows.
\end{enumerate}
\end{proof}

Finally, we deduce Theorem \ref{th:main_degY_2}. 

\begin{proof}{Proof of Theorem \ref{th:main_degY_2}:}
Take $d = e = \deg_X f$, then $\bar f = f_2(X) Y^2 + f_1(X) YZ + f_0(X)Z^2 \in \cC_{d, e}$
(note that we \emph{homogenize to degree} 2 even in the case $\deg_Y f = 0$). 
By Theorems \ref{th:extr_rays} and \ref{th:main},
$$
\bar f= \sum_{1 \leq i \leq s}r_i(p_iY+q_iZ)^2 
$$
for some $r_i,p_i,q_i\in \R[X]$  as in Theorem \ref{th:main} for $1 \le i \le s$.
By studying the factorization in $\C[X]$ of each $r_i \in \R[X]$, it is easy to see that 
the condition $r_i \ge 0$ on $[0, 1]$ implies that there exist
$t_i, u_i, v_i, w_i \in \sum \R[X]^2$
such that 
$$
r_i = t_i + u_iX + v_i(1-X) + w_iX(1-X)
$$
with $\deg t_i, \deg u_iX , \deg v_i(1-X), \deg  w_iX(1-X) \le \deg r_i$. 
Using the identities
$$
X = X^2 + X(1-X)
\qquad \hbox{ and } \qquad
1-X = (1-X)^2 + X(1-X),
$$
we take 
$$\sigma_0 = \sum_{1 \leq i \leq s}(t_i +   u_iX^2 + v_i(1-X)^2 )(p_iY+q_i)^2$$
and
$$\sigma_1 = \sum_{1 \leq i \leq s}(u_i  + v_i +  w_i) (p_iY+q_i)^2$$
and the identity $f = \sigma_0 + \sigma_1X(1-X)$ holds. 
Finally, 
$$
\deg (\sigma_0) \le \max_{1 \le i \le s} \deg(t_i +   u_iX^2 + v_i(1-X)^2 )(p_iY+q_i)^2 
\le \max_{1 \le i \le s} \deg  r_i(p_iY+q_i)^2 + 1 \le \deg_X f + 3
$$
and 
$$
\deg (\sigma_1X(1-X)) \le \max_{1 \le i \le s} \deg(u_i  + v_i +  w_i) (p_iY+q_i)^2X(1-X)
\le 
$$
$$
\le \max_{1 \le i \le s} \deg  r_i(p_iY+q_i)^2  +1 \le \deg_X f + 3.
$$
\end{proof}

\section{A constructive approach}\label{sec:al_app}

In this section
we show, under certain hypothesis, a constructive approach which also provides a degree bound for each term in the 
representation in Theorem \ref{th:Marshall}.
This approach works in the case that $f$ is positive on the strip and 
fully $m$-ic on $[0, 1]$ (Section \ref{subsect:positiv}) 
and in the case that $f$ is non-negative on the strip, fully $m$-ic on $[0, 1]$, and has only a finite 
number of zeros, all of them lying on the boundary of the strip and
such that $\frac{\partial f}{\partial x}$ does not vanish at any of them (Section \ref{subsect:zeros}). 
Finally, we will see in Example \ref{ex:caso_malo} that 
this approach does not work in the general case.

Roughly speaking, the main idea is to lift the interval $[0, 1]$ to the standard 1-dimensional simplex 
$$
\Delta_1 = \{(w, x) \in \R^2 \ | \ w \ge 0, \, x \ge 0, \, w+x = 1\},
$$
to consider $Y$ as a parameter and to produce for each evaluation of $Y$ 
a certificate of non-negativity on $\Delta_1$ 
using the effective version of
P\'olya's Theorem from \cite{Polya_bound} in a suitable manner so that these certificates can be glued together. 
We introduce a variable $W$ which is used to lift the interval $[0, 1]$ to the simplex $\Delta_1$ and, as before, a variable 
$Z$ which is used to compactify $\R$.

\begin{notn}
Given
$$
f= \sum_{0 \leq i \leq m} \sum_{0 \leq j \leq d} a_{ji}X^jY^i \in \R[X,Y],
$$
define
$$
F = \sum_{0 \leq i \leq m} \sum_{0 \leq j \leq d} a_{ji}X^j (W+X)^{d-j} Y^iZ^{m-i}
\in \R[W,X,Y,Z].
$$
For $N \in \N_0$ and $0 \le j \le N+d$, we define the polynomials $b_j \in \R[Y, Z]$ as follows:
\begin{equation} \label{eq:reemplazo_fund}
(W+X)^N F= \sum_{0 \leq j \leq N+d}b_j (Y,Z) W^j X^{N+d-j}.
\end{equation}
 \end{notn}
Note that  $(W+X)^NF$
is homogeneous on $(W,X)$ and $(Y,Z)$ of degree $N+d$ and $m$ respectively. 
Therefore, for $0 \le j \le N+d$, $b_j \in \R[Y, Z]$ is a homogeneous polynomial of degree $m$. 

We introduce the notation 
$$
C = \{(y, z) \in \R^2 \ | \ y^2 + z^2 = 1\}.
$$

\begin{proposition}\label{prop:metodo_gral} Let $f \in \R[X, Y]$ and 
$N \in \N_0$ such that for  $0 \le j \le N+d$, $b_j \ge 0$ on $C$.
Then $f$ can be written as in (\ref{repr_f}) with
$$
\deg(\sigma_0), \deg(\sigma_1 X(1-X)) \leq N + d + m +1.
$$
\end{proposition}

\begin{proof}{Proof:}
Substituting $W = 1-X$ and $Z = 1$ in  (\ref{eq:reemplazo_fund}) we have
$$
f(X, Y) =\sum_{0 \leq j \leq N+d}b_j (Y,1) (1-X)^j X^{N+d-j}.
$$
For $0 \le j \le N + d$, 
since $b_j(Y, Z) \ge 0$ on $C$ and $b_j$ is homogeneous, 
we have $b_j(Y,1) \geq 0$ on $\R$ and therefore  
$b_j(Y,1)$  is 
a sum of squares in $\R[Y]$ (see \cite[Proposition 1.2.1]{Marshall_book}) 
with the degree of each term bounded by $m$.

If $N+d$ is even, 
we take 
$$\sigma_0 = \sum_{0 \leq j \leq N+d, \ j \hbox{ {\small even}}}b_j(Y, 1)(1-X)^jX^{N+d-j}$$
and
$$\sigma_1 = \sum_{1 \leq j \leq N+d-1, \ j \hbox{ {\small odd}}}b_j(Y, 1)(1-X)^{j-1}X^{N+d-j-1}$$
and the identity $f = \sigma_0 + \sigma_1X(1-X)$ holds. 
In addition, we have
$$
\deg(\sigma_0),  \deg(\sigma_1X(1-X)) \le N + d + m.
$$

If $N+d$ is odd, 
using the identities
$$
X = X^2 + X(1-X)
\qquad \hbox{ and } \qquad
1-X = (1-X)^2 + X(1-X),
$$
we take 
$$
\sigma_0 = \sum_{0 \leq j \leq N+d-1, \ j \hbox{ {\small even}}}b_j(Y, 1)(1-X)^jX^{N+d-j+1} + 
\sum_{1 \leq j \leq N+d, \ j \hbox{ {\small odd}}}b_j(Y, 1)(1-X)^{j+1}X^{N+d-j}
$$
and
$$
\sigma_1 = \sum_{0 \leq j \leq N+d-1, \ j \hbox{ {\small even}}}b_j(Y, 1)(1-X)^jX^{N+d-j-1} + 
\sum_{1 \leq j \leq N+d, \ j \hbox{ {\small odd}}}b_j(Y, 1)(1-X)^{j-1}X^{N+d-j}
$$
and the identity $f = \sigma_0 + \sigma_1X(1-X)$ holds. 
In addition, we have
$$
\deg(\sigma_0),  \deg(\sigma_1X(1-X)) \le N + d + m + 1.
$$
\end{proof}

In Section \ref{subsect:positiv} and Section \ref{subsect:zeros}, under certain 
hypothesis, 
we prove the 
existence and find an upper bound for $N \in \N_0$ satisfying the hypothesis 
of Proposition \ref{prop:metodo_gral}. 
Then, to obtain the representation (\ref{repr_f})
we proceed as follows. 
If it possible to compute the upper bound, 
we compute the expansion of the polynomial $(W + X)^NF$
and then we compute the representation of each $b_j(Y, 1)$ as a sum of 
squares in $\R[Y]$ (see \cite{MagSafSch}).
If it is not possible to compute the upper bound, we 
pick a value of $N$ and 
we proceed by increasing $N$ one by one, we check symbolically 
at each step if it is the case that
$b_j(Y, 1)$ is non-negative on $\R$ for every $0 \le j \le N+d$ 
(see \cite[Chapter 4]{BPR} and \cite{PerRoy}), 
and once this condition is satisfied  
we compute the representation of each $b_j(Y, 1)$ as a sum of 
squares in $\R[Y]$.

For a homogeneous polynomial
$$
g = \sum_{0 \le j \le d} c_jW^jX^{d-j}\in \R[W, X]
$$
we note, as in \cite{Polya_bound}, 
$$
\| g \| = \max \left\{ \frac{|c_{j}|}{\binom{d}{j} }
\ | \  0 \le j \le d \right\}.
$$
One of the main tools we use is the effective version of P\'olya's Theorem from 
\cite{Polya_bound}. In the case of a
homogeneous polynomial $g \in \R[W, X]$ which is positive on $\Delta_1$, 
this theorem states that after 
multiplying for a suitable power of $W+X$, every coefficient is positive. 
Since we will need an explicit positive lower bound for these coefficients, 
we present in Lemma \ref{lem:polya_local}
a slight adaptation of \cite[Theorem 1]{Polya_bound}.  
We omit its proof since it can 
be developed exactly as the proof 
of \cite[Theorem 1]{Polya_bound} with only a minor modification at the final step.

\begin{lemma}\label{lem:polya_local}
Let $g \in \R[W,X]$ homogeneous of degree $d$ with $g>0$ on $\Delta_1$ and
let $\lambda=\min_{\Delta_1} g >0$. For $0 \leq \epsilon<1$, if
$$
N+d \geq \frac{(d-1)d\|g\|}{2(1-\epsilon) \lambda},
$$
for $0 \leq j \leq N+d$ the coefficient of $W^jX^{N+d-j}$ in $(W+X)^Ng$ is greater than or equal to 
$\frac{N!(N+d)^d}{j!(N+d-j)!} \epsilon \lambda$.
\end{lemma}

\subsection{The case of $f$ positive on the strip}\label{subsect:positiv}

In this section, we study the case of $f$ positive on $[0, 1] \times \R$
and fully $m$-ic on $[0, 1]$ and we prove Theorem \ref{thm:f_positivo}.

\begin{proposition}\label{prop:cota_para_N}
Let $f \in \R[X,Y]$ 
with $f > 0$ on $[0,1]\times \R$, $f$ fully $m$-ic on $[0, 1]$ and 
$$
f^{\bullet} =  \min\{\bar f(x, y, z) \ | \ x \in [0, 1], \, y^2 + z^2 = 1 \} > 0.
$$
Then, if
$$
N+d > \frac{(d-1)d(d+1)(m+1)\|f\|_{\infty}}{2 f^{\bullet}}, 
$$
for every $0 \le j \le N+d$, $b_j \ge 0$ on $C$. 
\end{proposition}

\begin{proof}{Proof:}
Since for every $(w,x,y,z) \in \Delta_1 \times C$, 
$F(w,x,y,z)=\bar f(x,y,z)$ we have $F \geq f^{\bullet} $ on $\Delta_1 \times C$.
   
On the other hand, it is easy to see that for $(y,z) \in C$, 
  $$
\|F(W,X,y,z)\|  \leq (d+1)(m+1) \max_{\substack{0 \leq i \leq m \\ 0 \leq j \leq d}} \left\{ \|a_{ji}X^j (W+X)^{d-j} y^iz^{m-i}  \| \right\}  
\le 
(d+1)(m+1) \|f\|_{\infty}.
$$
Using the bound for Polya's Theorem from \cite[Theorem 1]{Polya_bound},
if $N \in \N$ verifies
$$
N+d> \frac{(d-1)d(d+1)(m+1)\|f\|_{\infty}}{2 f^{\bullet}},
$$
all the coefficients of the polynomial
$$
(W+X)^N F(W,X,y,z) = \sum_{0 \le j \le N+d}b_j(y, z)W^jX^{N+d-j}  \in \R[W, X]
$$
are positive. In other words, for $0 \le j \le N +d$, $b_j \ge 0$ on $C$ as we wanted to prove. 
\end{proof}

We deduce easily Theorem \ref{thm:f_positivo}.
 
\begin{proof}{Proof of Theorem \ref{thm:f_positivo}:} 
By Proposition \ref{prop:cota_para_N} if $N \in \N$ is the smallest integer number such that 
$$
N+d > \frac{(d-1)d(d+1)(m+1)\|f\|_{\infty}}{2 f^{\bullet}},
$$
then  for every $0 \le j \le N+d$, $b_j \ge 0$ on $C$.
By Proposition \ref{prop:metodo_gral}, we have that $f$ can be written as in (\ref{repr_f}) with 
$$
\deg(\sigma_0), \deg(\sigma_1 X(1-X)) \leq N+d+m+1.
$$

Since 
$$
  \|f\|_{\infty}\geq |a_{00}|=|f (0,0)| = f(0, 0) =\bar f (0,0,1) \geq f^{\bullet},
  $$
we have
$$
    \deg(\sigma_0), \deg(\sigma_1 X(1-X)) \leq N+d+m+1 
     \leq \frac{(d-1)d(d+1)(m+1)\|f\|_{\infty}}{2 f^{\bullet}} +m+2 \le 
    \frac{d^3(m+1)\|f\|_{\infty}}{ f^{\bullet}}.
$$
  \end{proof}

\subsection{The case of $f$ with a finite number of zeros on the boundary of the strip}\label{subsect:zeros}

Next, we want to relax the hypothesis
$f > 0$ on $[0,1] \times \R$ to
$f \ge 0$ on $[0,1] \times \R$ and with a finite numbers of zeros on the boundary of the strip. 
Consider 
$$
 C_+ = \{(y, z) \in \R^2 \ | \ y^2 + z^2 = 1, \ z \geq 0 \}.
$$
For $f$ non-negative in $[0, 1] \times \R$ and fully $m$-ic on $[0, 1]$, it is clear that $m$ is even. Then, since
each $b_j(Y, Z) \in \R[Y, Z]$ is homogeneous of degree $m$, to prove that   
$b_j \ge 0$ on $C$
it is enough to prove that 
$b_j \ge 0$ on $C_+$. The advantage of considering $C_+$ instead of $C$ is simply that 
under the present hypothesis
there is a bijection between
the zeros of $f$ in $[0, 1] \times \R$ and the zeros of $F$ in $\Delta_1 \times C_+$ given by
$$
(x, \alpha)  \mapsto  (1-x,x, y_{\alpha}, z_{\alpha}) \qquad \hbox { with } \qquad 
(y_{\alpha}, z_{\alpha}) = \left(\frac{\alpha}{\sqrt{\alpha^2+1}}, \frac{1}{\sqrt{\alpha^2+1}} \right). 
$$

The idea is to consider separately, 
for each  zero $(x, \alpha)$ of $f$,  
the polynomial $F(W,X,y_{\alpha},z_{\alpha}) \in \R[W,X]$ and
to find $N_{\alpha} \in \N_0$
such that $(W+X)^{N_{\alpha}}F(W,X,y_{\alpha},z_{\alpha})$ has non-negative coefficients
$b_j(y_\alpha, z_\alpha)$.
Then, we show that the same $N_\alpha$
works for $(y,z) \in C_+$ close to $(y_{\alpha},z_{\alpha})$. Finally, 
in the rest of $C_+$ we  use compactness arguments.

\begin{proposition}\label{prop:cota_para_N_con_ceros_bordes}
  Let $f \in \R[X,Y]$ with
  $f \geq 0$ on $[0,1]\times \R$, $f$ fully $m$-ic on $[0, 1]$ and
  suppose that $f$ has a finite number of zeros in $[0,1]\times\R$, all of them lying on $\{0,1\}\times\R$, 
  and   $\frac{\partial f}{\partial X}$ 
  does not vanish at any of them. 
  Then, there is  $N \in \N_0$ such that
  for every $0 \le j \le N+d$, $b_j \ge 0$ on $C$. 
\end{proposition} 

\begin{proof}{Proof:}
For $0 \le h \le d$, we define the polynomials $c_h \in \R[Y, Z]$ as follows: 
$$
F = \sum_{0 \le h \le d} c_h(Y, Z)W^hX^{d-h}.
$$
Then, for $0 \le h \le d$, 
$$
c_h(Y,Z)= \sum_{0 \leq i \leq m} \sum_{0 \leq j \leq d-h}   
a_{ji}\binom{d-j}{h}  Y^iZ^{m-i}
$$
is a homogeneous polynomial in $\R[Y, Z]$ of 
degree $m$, and for $(y, z) \in C_+$ we have 
\begin{equation}\label{eq:cota_c_h}
|c_h(y,z)|\le (m+1)\|f \|_{\infty} \sum_{0 \leq j \leq d-h}\binom{d-j}{h}
=  
(m+1)\binom{d+1}{h+1}\|f \|_{\infty} 
\end{equation}
and
\begin{equation}\label{eq:cota_norma_F}
\|F(W, X, y, z)\| \le 
\max \left \{ (m+1) \frac{\binom{d+1}{h+1}}{\binom{d}{h}}\|f \|_{\infty} \ | \ 0 \le h \le d   \right\} \le 
(m+1) (d+1)\|f \|_{\infty}.
\end{equation}

Now, since along the proof we will consider several values of $N$, 
we add the index $N$ to the notation of polynomials $b_j$ in the following way: 
$$
(W+X)^N F= \sum_{0 \leq j \leq N+d}b_{j,N} (Y,Z) W^j X^{N+d-j}.
$$
So we need to prove that there is $N \in \N_0$ such that 
for every $0 \le j \le N+d$, $b_{j,N} \ge 0$ on $C_+$. 
It is clear that, for a fixed $(y, z) \in C_+$, if $N \in \N_0$ satisfies that  
for every $0 \le j \le N+d$, $b_{j,N}(y, z) \ge 0$, then  
any $N' \in \N_0$ with $N' \ge N$ also satisfies 
that for every $0 \le j \le N'+d$, $b_{j,N'}(y, z) \ge 0$.

For $N \in \N_0$ and $\alpha \in \R$, we have the identities
\begin{equation} \label{eq:f_en_1}
b_{0,N}(y_\alpha, z_\alpha) = 
c_0(y_\alpha, z_\alpha) = 
F(0, 1, y_\alpha, z_\alpha) = 
\bar f(1, y_\alpha, z_\alpha) = 
\frac{1}{\sqrt{\alpha^2 + 1}^m}f(1, \alpha) 
\end{equation}
and 
\begin{equation} \label{eq:f_en_0}
b_{N+d,N}(y_\alpha, z_\alpha) =   
c_d(y_\alpha, z_\alpha) =
F(1, 0, y_\alpha, z_\alpha) = 
\bar f(0, y_\alpha, z_\alpha) = 
\frac{1}{\sqrt{\alpha^2 + 1}^m}f(0, \alpha).  
\end{equation}
From (\ref{eq:f_en_1}) and (\ref{eq:f_en_0}) we deduce that 
for every $N \in \N_0$, $b_{0, N}
\ge 0$ on $C_+$ and $b_{N+d, N} \ge 0$ on $C_+$. 
So we need to prove that there is $N \in \N_0$ such that 
for every $1 \le j \le N+d-1$, $b_{j,N} \ge 0$ on $C_+$. 

We note
$$
\Pi_f= \{\alpha \in \R \ | \  f(x,\alpha)=0 \ \mbox{for some} \ x \in \{0,1\} \} \subseteq \R.
$$
We will show first that for each $\alpha \in \Pi_f$ there is  
$N_{\alpha}\in \N_0$ such that for $1 \le j \le N_\alpha+d-1$, 
$b_{j, N_\alpha}(y_{\alpha},z_{\alpha})$ is positive on $C_+$.
We consider three cases:

\begin{itemize}
\item  $f(0,{\alpha})=0$ and $f(1, \alpha) \ne 0$: 

From (\ref{eq:f_en_0}) we have 
$b_{N+d, N}(y_\alpha, z_\alpha) = 0$ for every $N \in \N_0$ and 
also $c_d(y_\alpha, z_\alpha) = 0$. 
We consider the homogeneous polynomial of degree $d-1$
$$
\widetilde F_{{\alpha}}(W,X)= \frac{F(W,X,y_{{\alpha}},z_{{\alpha}})}
{X}
= \sum_{0 \le h \le d-1} c_h(y_{\alpha}, z_{\alpha})W^hX^{d-h-1}
\in \R[W,X].
$$
From (\ref{eq:cota_c_h}) we deduce 
that for $0 \le h \le d-1$,
$$
\frac{|c_h(y_\alpha, z_\alpha)|}{\binom{d-1}{h}} \le 
(m+1)\frac{\binom{d+1}{h+1}}{\binom{d-1}{h}}\|f \|_{\infty} =
(m+1)\frac{(d+1)d}{(h+1)(d-h)}\|f \|_{\infty}  \le 
(m+1)(d+1)\|f \|_{\infty}
$$
and we have
$\|\tilde F_{\alpha}\| \le (m+1)(d+1)\|f \|_{\infty}$.

On the other hand, it is clear that 
$\widetilde F_{{\alpha}}>0$ on $\Delta_1-\{(1,0)\}$ and, in addition, 
$$
\widetilde F_{{\alpha}}(1,0) = \frac{\partial F(1, 0, y_{\alpha}, z_{\alpha})}{\partial X}
= 
\frac{1}{\sqrt{{\alpha}^2+1}^m} \frac{\partial  f}{\partial X}(0,{\alpha})  
> 0
 $$
 therefore $\widetilde F_{{\alpha}}(1,0)>0$.  We note
 $$
 \lambda_{{\alpha}}= \min_{\Delta_1}\widetilde F_{{\alpha}}>0.
 $$

 By Lemma \ref{lem:polya_local} with $\epsilon = 1/2$, if
 $N_{\alpha} \in \N_0$ satisfies
 $$
 N_{\alpha}+d -1 \geq \frac{(d-2)(d-1)(d+1)(m+1)\|f\|_{\infty}}{\lambda_{\alpha}}
 $$
and
$$
(W+X)^{N_{\alpha}}\widetilde F_{{\alpha}} = \sum_{0 \le j \le N_{\alpha} + d-1}
c_{j}W^{j}X^{N_{\alpha}+d-1-j},
$$
for $0 \leq j \leq N_{\alpha}+d-1$ we have 
$$
c_{j} \ge \frac{N_{\alpha}!(N_{\alpha}+d-1)^{d-1}}{j!(N_{\alpha}+d-1-j)!} 
 \frac{\lambda_{\alpha}}{2}.$$
But since
$$
\sum_{0 \le j \le N_{\alpha} + d-1}
c_{j}W^{j}X^{N_{\alpha}+d-j} = 
(W+X)^{N_{\alpha}}X\widetilde F_{{\alpha}} =
$$
$$
=
(W+X)^{N_{\alpha}}F(W, X, y_\alpha, z_\alpha) 
  = \sum_{0 \le j \le N_\alpha+d-1}b_{j, N_\alpha}(y_\alpha, z_\alpha)W^jX^{N_\alpha+d-j}
$$
we conclude that for $0 \le j \le N_\alpha + d-1$, 
$$
b_{j, N_\alpha}(y_\alpha, z_\alpha)  = c_{j}\ge 
\frac{N_{\alpha}!(N_{\alpha}+d-1)^{d-1}}{j!(N_{\alpha}+d-1-j)!} 
 \frac{\lambda_{\alpha}}{2}.
$$

\item  $f(0,{\alpha})\ne0$ and $f(1, \alpha) = 0$:

From (\ref{eq:f_en_1}) we have 
$b_{0, N}(y_\alpha, z_\alpha) = 0$ for every $N \in \N_0$ and 
also $c_0(y_\alpha, z_\alpha) = 0$. 
We consider the homogeneous polynomial of degree $d-1$
$$
\widetilde F_{{\alpha}}(W,X)= \frac{F(W,X,y_{{\alpha}},z_{{\alpha}})}
{W}
= \sum_{1 \le h \le d} c_h(y_{\alpha}, z_{\alpha})W^{h-1}X^{d-h}
\in \R[W,X].
$$
Then, proceeding similarly to the previous case we prove
$\|\tilde F_{\alpha}\| \le \frac12(m+1)d(d+1)\|f \|_{\infty}$.
Moreover, since 
$\widetilde F_{{\alpha}}>0$ on $\Delta_1-\{(0,1)\}$ and  
$$
\widetilde F_{{\alpha}}(0,1) = \frac{\partial F(0, 1, y_{\alpha}, z_{\alpha})}{\partial W}
= 
- \frac{1}{\sqrt{{\alpha}^2+1}^m} \frac{\partial  f}{\partial X}(1,{\alpha})  
> 0
 $$
we have that  $\widetilde F_{{\alpha}}(1,0)>0$ and  we note
 $$
 \lambda_{{\alpha}}= \min_{\Delta_1}\widetilde F_{{\alpha}}>0.
 $$
Finally, using Lemma \ref{lem:polya_local} with $\epsilon = 1/2$,
we conclude that 
if
 $N_{\alpha} \in \N_0$ satisfies
 $$
 N_{\alpha}+d -1 \geq \frac{(d-2)(d-1)d(d+1)(m+1)\|f\|_{\infty}}{2\lambda_{\alpha}},
 $$
for $1 \le j \le N_\alpha + d$, 
$$
b_{j, N_\alpha}(y_\alpha, z_\alpha) \ge 
\frac{N_{\alpha}!(N_{\alpha}+d-1)^{d-1}}{(j-1)!(N_{\alpha}+d-j)!} 
 \frac{\lambda_{\alpha}}{2}.
$$

\item $f(0,{\alpha})=0$ and $f(1, \alpha) = 0$:

From (\ref{eq:f_en_1}) and (\ref{eq:f_en_0}) we have 
$b_{0, N}(y_\alpha, z_\alpha) = b_{N+d, N}(y_\alpha, z_\alpha) = 0$ for every $N \in \N_0$ and 
also $c_0(y_\alpha, z_\alpha) = c_d(y_\alpha, z_\alpha) = 0$. 
We consider the homogeneous polynomial of degree $d-2$
$$
\widetilde F_{{\alpha}}(W,X)= \frac{F(W,X,y_{{\alpha}},z_{{\alpha}})}
{WX}
= \sum_{1 \le h \le d-1} c_h(y_{\alpha}, z_{\alpha})W^{h-1}X^{d-h-1}
\in \R[W,X].
$$
Then, proceeding similarly to the previous cases we prove again
$\|\tilde F_{\alpha}\| \le \frac12(m+1)d(d+1)\|f \|_{\infty}$.
We note
 $$
 \lambda_{{\alpha}}= \min_{\Delta_1}\widetilde F_{{\alpha}}>0.
 $$
Finally, using Lemma \ref{lem:polya_local} with $\epsilon = 1/2$,
we conclude that 
if
 $N_{\alpha} \in \N_0$ satisfies
 $$
 N_{\alpha}+d -2 \geq \frac{(d-3)(d-2)d(d+1)(m+1)\|f\|_{\infty}}{2\lambda_{\alpha}},
 $$
for $1 \le j \le N_\alpha + d-1$, 
$$
b_{j, N_\alpha}(y_\alpha, z_\alpha) \ge 
\frac{N_{\alpha}!(N_{\alpha}+d-2)^{d-2}}{(j-1)!(N_{\alpha}+d-1-j)!} 
 \frac{\lambda_{\alpha}}{2}.
$$
\end{itemize}

Now, our next goal is to compute a radios $r_\alpha > 0$ around each $(y_\alpha, z_\alpha)$ so that
for $(y, z) \in C_+$ with $\| (y, z) - (y_\alpha, z_\alpha) \| \le r_\alpha$, 
for $1 \le j \le N_\alpha + d - 1$, we have 
$b_{j, N_\alpha}(y, z) \ge 0$.
First, we do some auxiliary computations.

For $0 \le h \le d$ and $(y, z) \in \R^2$ with $y^2 + z^2 \le 1$ we have  
\begin{align*}
 \|\nabla c_h(y, z) \| 
 & \leq \left| \frac{\partial c_h}{\partial Y}(y, z) \right| + \left| \frac{\partial c_h}{\partial Z}(y, z) \right|\\
 & \leq  \sum_{1 \leq i \leq m} \sum_{0 \leq j \leq d-h} |a_{ji}| \binom{d-j}{h}i +  \sum_{0 \leq i \leq m-1} \sum_{0 \leq j \leq d-h} |a_{ji}| \binom{d-j}{h}(m-i)\\
 & \le  m(m+1)   \binom{d+1}{h+1}\|f\|_{\infty}  \\
 & \leq m(m+1) (d+1) \binom{d}{h}  \|f\|_{\infty}.
\end{align*}
Then, for $(y, z) \in C_+$, 
$$
|c_h(y,z)-c_h(y_{\alpha},z_{\alpha})| \leq m(m+1) (d+1) \binom{d}{h}\|f\|_{\infty} \|(y,z)-(y_{\alpha},z_{\alpha})\|.
$$

We introduce now some notation following \cite{Polya_bound}.
For $t \in \R$, $m\in \N_0$ and a variable $U$, 
$$
(U)_t^m:=U(U-t)(U-2t)\cdots(U-(m-1)t)= \prod_{0\leq i \leq m-1}(U-it) \in \R[U].
$$
Also, for $t \in \R$ 
$$
F_{t}(W, X, Y, Z) =  \sum_{0 \le h \le d}c_h(Y, Z)(W)_t^h(X)_t^{d-h}.
$$

By \cite[(4)]{Polya_bound}, for $N \in \N_0$ and  $0 \le j \le N+d$ we have
$$
b_{j, N}(y,z)= \frac{N!(N+d)^d}{j!(N+d-j)!} F_{\frac{1}{N+d}}\left(\frac{j}{N+d}, \frac{N+d-j}{N+d},y,z \right).
$$
Then, using the Vandermonde-Chu identity (see \cite[(6)]{Polya_bound}), for $(y, z) \in C_+$ we have 
\begin{eqnarray*}
& & \left |  F_{\frac{1}{N+d}}\left(\frac{j}{N+d},\frac{N+d-j}{N+d},y,z \right) - 
F_{\frac{1}{N+d}}\left(\frac{j}{N+d},\frac{N+d-j}{N+d},y_{\alpha},z_{\alpha} \right) \right |  
\\[3mm]
& & \le 
 \sum_{0 \leq h \leq d} | c_h(y,z)-c_h(y_{\alpha},z_{\alpha}) | 
 \left(\frac{j}{N+d} \right)^h_{\frac{1}{N+d}} 
 \left(\frac{N+d-j}{N+d} \right)^{d-h}_{\frac{1}{N+d}}
\\[3mm]
& & \le 
 m(m+1)(d+1)\|f\|_{\infty}
 \|(y,z)-(y_{\alpha},z_{\alpha})\|
 \left(
 \sum_{0 \leq h \leq d}  
 \binom{d}{h}
 \left(\frac{j}{N+d} \right)^h_{\frac{1}{N+d}} 
 \left(\frac{N+d-j}{N+d} \right)^{d-h}_{\frac{1}{N+d}}
 \right)
\\[3mm]
& & =
 m(m+1)(d+1)\|f\|_{\infty}
 \|(y,z)-(y_{\alpha},z_{\alpha})\|
 (1)^d_{\frac{1}{N+d}}
\\[3mm]
& & \le 
m(m+1)(d+1)\|f\|_{\infty}
 \|(y,z)-(y_{\alpha},z_{\alpha})\|.
\end{eqnarray*}

Consider $\alpha \in \Pi_f$.  If
$f(0, \alpha) = 0$ and $f(1, \alpha) \ne 0$
we take 
$$
r_{\alpha}=\frac{\lambda_{\alpha}(N_{\alpha}+d-1)^{d-1}}{2(N_{\alpha}+d)^dm(m+1)(d+1)\|f\|_{\infty}}.
$$
Then, for $(y, z) \in C_+$ with $\|(y, z) - (y_\alpha, z_\alpha)\| \le r_\alpha$  and $1 \le j \le N_\alpha + d - 1$ we have 
\begin{eqnarray*}
b_{j, N}(y,z) &= & b_{j, N}(y_{\alpha},z_{\alpha})+b_{j, N}(y,z)-b_{j, N}(y_{\alpha},z_{\alpha}) \\[3mm]
& \geq &
\frac{N_{\alpha}!(N_{\alpha}+d-1)^{d-1}}{j!(N_{\alpha}+d-1-j)!} 
 \frac{\lambda_{\alpha}}{2}
 -  \frac{N_\alpha!(N_\alpha+d)^d}{j!(N_\alpha+d-j)!} m(m+1)(d+1)\|f\|_{\infty} r_\alpha \\[3mm]
& \geq & 0.
\end{eqnarray*}
If 
$f(0, \alpha) \ne 0$ and $f(1, \alpha) = 0$ 
we take again
$$
r_{\alpha}=\frac{\lambda_{\alpha}(N_{\alpha}+d-1)^{d-1}}{2(N_{\alpha}+d)^dm(m+1)(d+1)\|f\|_{\infty}}
$$
and if 
$f(0, \alpha) \ne 0$ and $f(1, \alpha) = 0$ 
we take 
$$
r_{\alpha}=\frac{\lambda_{\alpha}(N_{\alpha}+d-2)^{d-2}}{2(N_{\alpha}+d)^dm(m+1)(d+1)\|f\|_{\infty}}
$$
and in both cases we proceed in a similar way.

Now, consider $K \subseteq C_+$ defined by
$$
K=\left\{(y,z) \in C_+ : \|(y,z)-(y_{\alpha},z_{\alpha})\| \geq r_{\alpha} \ \mbox{for all} \ {\alpha} \in \Pi_f \right\}.
$$
Since $K$ is compact and $\lambda_K= \min_{\Delta_1 \times K} F >0$, by \cite[Theorem 1]{Polya_bound} using (\ref{eq:cota_norma_F}),
if
$$
N+d> \frac{(d-1)d(d+1)(m+1)\|f\|_{\infty}}{2 \lambda_K},
$$
for  $0 \le j \le N+d$,
$b_{j,N}(y,z)\geq 0$ for every $(y,z) \in K$.

Finally, if $N \in \N$,  
$$
N = \max\left\{\Big\lfloor \frac{(d-1)d(d+1)(m+1)\|f\|_{\infty}}{2 \lambda_K} \Big \rfloor -d+1, \, \max \left \{ N_{\alpha} \, | \,  {\alpha} \in \Pi_f \right \} \right\},
$$ 
we conclude that for $0 \le j \le N+d$, $b_{j, N} \geq 0$ on $C_+$.
\end{proof}

From Proposition \ref{prop:metodo_gral} and Proposition \ref{prop:cota_para_N_con_ceros_bordes} we deduce the 
following result.

\begin{theorem}\label{thm:cota_con_ceros_bordes}
Let $f \in \R[X,Y]$ with
$f \geq 0$ on $[0,1]\times \R$, $f$ fully $m$-ic on $[0, 1]$ and
suppose that $f$ has a finite number of zeros in $[0,1]\times\R$, all of them lying on $\{0,1\}\times\R$, 
and   $\frac{\partial f}{\partial X}$
does not vanish at any of them. 
Then, for $N \in \N_0$ as in  Proposition \ref{prop:cota_para_N_con_ceros_bordes},
$f$ can be written as in (\ref{repr_f}) with 
$$
\deg(\sigma_0), \deg(\sigma_1 X(1-X)) \leq N+d+m+1.
$$
\end{theorem}

We conclude with an example of a polynomial 
$f \in \R[X,Y]$ with
$f \geq 0$ on $[0,1]\times \R$, $f$ fully $m$-ic on $[0, 1]$, 
with only one zero in $[0,1]\times\R$ lying on $\{0,1\}\times\R$ 
but $\frac{\partial f}{\partial X}$ vanishing at it, and such that $f$
does not admit a value of $N \in \N_0$ as in Proposition \ref{prop:metodo_gral}. 
Note that in this example, $f$ is itself a sum of
squares, so the representation as in (\ref{repr_f}) is already given; nevertheless, our purpose is to show that 
there is no hope of applying the method underlying Proposition \ref{prop:metodo_gral} in full generality.

\begin{example}\label{ex:caso_malo}
Let
$$
f(X,Y)= (Y^2-X)^2+X^2=Y^4-2XY^2+2X^2.
$$
Then
$$
F(W,X,Y,Z)=(W+X)^2Y^4-2X(W+X)Y^2Z^2+2X^2Z^4.
$$
and for $N \in \N$, 
$$
(W+X)^NF(W,X,Y,Z) = 
Y^4 W^{N+2} + Y^2 \left( (N+2)Y^2-2Z^2 \right) W^{N+1}X + \dots
$$
It is easy to see that it does not exist $N \in \N_0$ such that
$$
b_{N+1}(Y,Z)=Y^2 \left( (N+2)Y^2-2Z^2 \right)
$$
is non-negative on $C$. 
\end{example}

\textbf{Acknowledgments:} We want to thank the anonymous referee for her/his helpful suggestions.


\begin{thebibliography}{10}

\bibitem{BPR}
Saugata Basu, Richard Pollack, and Marie-Fran\c{c}oise Roy.
\newblock {\em Algorithms in real algebraic geometry}, volume~10 of {\em
  Algorithms and Computation in Mathematics}.
\newblock Springer-Verlag, Berlin, second edition, 2006.

\bibitem{EscPer}
Paula Escorcielo and Daniel Perrucci.
\newblock A version of {P}utinar's {P}ositivstellensatz for cylinders.
\newblock {\em J. Pure Appl. Algebra}, 224(12):106448, 16, 2020.

\bibitem{JacPres}
Thomas Jacobi and Alexander Prestel.
\newblock Distinguished representations of strictly positive polynomials.
\newblock {\em J. Reine Angew. Math.}, 532:223--235, 2001.

\bibitem{MagSafSch}
Victor Magron, Mohab Safey El~Din, and Markus Schweighofer.
\newblock Algorithms for weighted sum of squares decomposition of non-negative
  univariate polynomials.
\newblock {\em J. Symbolic Comput.}, 93:200--220, 2019.

\bibitem{Marshall_book}
Murray Marshall.
\newblock {\em Positive polynomials and sums of squares}, volume 146 of {\em
  Mathematical Surveys and Monographs}.
\newblock American Mathematical Society, Providence, RI, 2008.

\bibitem{Marsh}
Murray Marshall.
\newblock Polynomials non-negative on a strip.
\newblock {\em Proc. Amer. Math. Soc.}, 138(5):1559--1567, 2010.

\bibitem{NguPowers}
Ha~Nguyen and Victoria Powers.
\newblock Polynomials non-negative on strips and half-strips.
\newblock {\em J. Pure Appl. Algebra}, 216(10):2225--2232, 2012.

\bibitem{PerRoy}
Daniel Perrucci and Marie-Fran\c{c}oise Roy.
\newblock A new general formula for the {C}auchy {I}ndex on an interval with
  {S}ubresultants.
\newblock To appear in \textit{J. Symbolic Comput}.

\bibitem{Pow}
Victoria Powers.
\newblock Positive polynomials and the moment problem for cylinders with
  compact cross-section.
\newblock {\em J. Pure Appl. Algebra}, 188(1-3):217--226, 2004.

\bibitem{Polya_bound}
Victoria Powers and Bruce Reznick.
\newblock A new bound for {P}\'{o}lya's theorem with applications to
  polynomials positive on polyhedra.
\newblock {\em J. Pure Appl. Algebra}, 164(1-2):221--229, 2001.
\newblock Effective methods in algebraic geometry (Bath, 2000).

\bibitem{PowRez}
Victoria Powers and Bruce Reznick.
\newblock Polynomials positive on unbounded rectangles.
\newblock In {\em Positive polynomials in control}, volume 312 of {\em Lect.
  Notes Control Inf. Sci.}, pages 151--163. Springer, Berlin, 2005.

\bibitem{Put}
Mihai Putinar.
\newblock Positive polynomials on compact semi-algebraic sets.
\newblock {\em Indiana Univ. Math. J.}, 42(3):969--984, 1993.

\bibitem{Rockafellar}
Ralph~Tyrrell Rockafellar.
\newblock {\em Convex analysis}.
\newblock Princeton Mathematical Series, No. 28. Princeton University Press,
  Princeton, N.J., 1970.

\bibitem{Sche}
Claus Scheiderer.
\newblock Sums of squares of regular functions on real algebraic varieties.
\newblock {\em Trans. Amer. Math. Soc.}, 352(3):1039--1069, 2000.

\bibitem{SchWen}
Claus Scheiderer and Sebastian Wenzel.
\newblock Polynomials nonnegative on the cylinder.
\newblock In {\em Ordered algebraic structures and related topics}, volume 697
  of {\em Contemp. Math.}, pages 291--300. Amer. Math. Soc., Providence, RI,
  2017.

\bibitem{Schm}
Konrad Schm\"{u}dgen.
\newblock The {$K$}-moment problem for compact semi-algebraic sets.
\newblock {\em Math. Ann.}, 289(2):203--206, 1991.

\bibitem{Ste}
Gilbert Stengle.
\newblock Complexity estimates for the {S}chm\"{u}dgen {P}ositivstellensatz.
\newblock {\em J. Complexity}, 12(2):167--174, 1996.

\end{thebibliography}
\end{document}